\documentclass[11pt]{article}

\usepackage{tikz}
\usepackage{graphicx}
\usepackage{verbatim}
\usepackage{amssymb,amsfonts,amsmath,amsthm}
\usepackage{mathtools}
\usepackage{url}
\usepackage{fullpage}
\usepackage[citecolor=red]{hyperref}
\usepackage[capitalise]{cleveref}
\usepackage[english]{babel}
\usepackage{algorithm}
\usepackage{algpseudocode}
\usepackage{paralist}

\newcommand{\ceil}[1]{{\lceil#1\rceil}}

\newcommand{\N}{{\mathbb N}}

\newcommand{\Z}{{\mathbb Z}}

\newcommand{\seq}{\subseteq}

\renewcommand{\int}{{\sf int}}

\newcommand{\D}{{\mathcal D}}

\renewcommand{\epsilon}{\varepsilon}

\renewcommand{\sp}{\mathrm{span}}
\newcommand{\Cayley}{\mathcal{C}}
\newcommand{\GF}{\mathrm{GF}}
\newcommand{\Forbidden}{{F}}

\newtheorem{ftheorem}{Theorem}

\newtheorem{theorem}{Theorem}[section]

\newtheorem{lemma}[theorem]{Lemma}
\newtheorem{claim}[theorem]{Claim}
\newtheorem{observation}[theorem]{Observation}
\newtheorem{notation}[theorem]{Notation}

\newtheorem{question}[theorem]{Question}

\newenvironment{proof-sketch}{\noindent{\bf Sketch of Proof}\hspace*{1em}}{\qed\bigskip}
\newenvironment{proof-idea}{\noindent{\bf Proof Idea}\hspace*{1em}}{\qed\bigskip}
\newenvironment{proof-of-lemma}[1]{\noindent{\bf Proof of Lemma #1}\hspace*{1em}}{\qed\bigskip}
\newenvironment{proof-of-claim}[1]{\noindent{\bf Proof of Claim #1}\hspace*{1em}}{\qed\bigskip}
\newenvironment{proof-of-thm}[1]{\noindent{\bf Proof of Theorem #1}\hspace*{1em}}{\qed\bigskip}
\newenvironment{proof-attempt}{\noindent{\bf Proof Attempt.}\hspace*{1em}}{\qed\bigskip}

\title{Meyniel Extremal Families of Abelian Cayley Graphs}

\author{
    Fatemeh Hasiri\thanks{
    School of Computing Science, Simon Fraser University.
    Email: \texttt{fhasiri@sfu.ca}
    }
    \and
	Igor Shinkar\thanks{
    School of Computing Science, Simon Fraser University.
    Email: \texttt{ishinkar@sfu.ca}
    }
}

\begin{document}
	
\maketitle

\begin{abstract}
    We study the game of \emph{Cops and Robbers}, where cops try to capture a robber on the vertices of a graph.
    Meyniel's conjecture states that for every connected graph $G$ on $n$ vertices, the cop number of $G$ is upper bounded by $O(\sqrt{n})$,
    i.e., that $O(\sqrt{n})$ suffice to catch the robber.
    We present several families of abelian Cayley graphs that are Meyniel extremal, i.e., graphs whose cop number is $O(\sqrt{n})$.
    This proves that the $O(\sqrt{n})$ upper bound for Cayley graphs proved by Bradshaw~\cite{Bradshaw19} is tight up to a multiplicative constant.
    In particular, this shows that Meyniel's conjecture, if true, is tight to a multiplicative constant even for abelian Cayley graphs.

    In order to prove the result, we construct Cayley graphs on $n$ vertices with $\Omega(\sqrt{n})$ generators that are $K_{2,3}$-free.
    This shows that the K\"{o}v\'{a}ri, S\'{o}s, and Tur\'{a}n theorem, stating that any $K_{2,3}$-free graph of $n$ vertices has at most $O(n^{3/2})$ edges,
    is tight up to a multiplicative constant even for abelian Cayley graphs.
\end{abstract}

\section{Introduction}\label{sec:intro}

\emph{Cops and robber} is a two player game played on an undirected, finite, simple, and connected graph $G=(V, E)$.
The first player, called the \emph{cops player}, has $c$ cops, and second player, \emph{the robber}, has $1$ robber.
The game starts with the first player placing each cop in a vertex in $G$; then, the second player chooses the initial vertex for the robber.
The players play in alternate rounds, where in each turn of the cops the first player moves each cop along an edge to an adjacent vertex or keeps it in its current position,
and in robber's turn the robber may move along an edge to an adjacent vertex or stay in place.
The cops win if after some finite number of rounds, one of the cops \emph{captures} the robber by occupying the same vertex as the robber.
Otherwise, if the robber can avoid this situation forever, the robber wins the game.
The minimum value of $c$ for which the cops have a winning strategy is called the \emph{cop number of $G$}, and is denoted by $c(G)$.
The game of cops and robbers was first introduced by Nowakowski and Winkler~\cite{Nowakowski83}, and independently by Quilliot~\cite{quilliot1983problemes} for the case of $c=1$ cop, and later generalized by Aigner and Fromme~\cite{fromme1984game} to several cops.

Meyniel's conjecture, mentioned in Frankl's paper~\cite{Frankl1987a}, states that for any connected $n$-vertex graph $G$ it holds that $c(G) = O(\sqrt{n})$.
Despite considerable attention this problem has received recently, the conjecture remain open.
The best known upper bound, proved independently by \cite{lu2012meyniel,ScottS11,FriezeKL12}, says that the cop number of any graph on $n$ vertices is upper bounded by $n/2^{(1+o(1))\sqrt{n}}$. Sharper results are known for special classes of graphs, such as random graphs~\cite{BollobasKL13,BonatoPW07a,BonatoPW07b,LuczakP10,Pralat10}, planar graphs \cite{fromme1984game}, graphs with bounded genus \cite{Quilliot85, Schroder2001}, Cayley graphs \cite{Bradshaw19,Frankl1987}, and more. For a survey of known related results see~\cite{BN09book}.

There are several works in the literature~\cite{Pralat10, BairdBonato12, BonatoBurgess13} describing Meyniel extremal families of graphs,
i.e., families of graphs whose cop number is $\Omega(\sqrt{n})$ where $n$ is the number of vertices in the graph.
Our work contributes new examples of Meyniel extremal families.
Specifically,we present several Meyniel extremal families of \emph{abelian Cayley} graphs.

Informally, abelian Cayley graphs are very structured, symmetric graphs.
More formally, let $G$ be a finite group, and let subset $S$ be a symmetric subset of $G$, i.e., satisfying the property that if $a \in S$, then $-a \in S$.
The Cayley graph associated with $(G,S)$, denoted by $\Cayley(G,S)$,
is the graph whose vertices are the elements of $G$, and there is an edge between $g$ and $h$ if and only if $g-h \in S$.
We say that a Cayley graph $\Cayley(G,S)$ is abelian if the underlying group $G$ is abelian.

Frankl~\cite{Frankl1987} proved that for any  connected abelian Cayley graphs it holds that $c(\Cayley(G,S))\leq \ceil{(|S|+1)/2}$.
Recently, Bradshaw~\cite{Bradshaw19} showed that the cop number of any connected abelian Cayley graph on $n$ vertices is bounded by $7\sqrt{n}$.
In this work we prove a lower bound that matches Bradshaw's result up to a multiplicative constant.
In particular, if Meyniel's conjecture is true, then it is tight to a multiplicative constant even for abelian Cayley graphs.

\subsection{Our results}\label{sec:results}

In this paper we present several examples of Meyniel extremal families of abelian Cayley graphs,
i.e., abelian Cayley graphs on $n$ vertices whose cop number is $\Omega(\sqrt{n})$.

\begin{ftheorem}\label{thm:sqrt-lb}
The following graph families are Meyniel extremal.
\begin{enumerate}
    \item \label{item:Cayley-Z_n}
    Let $n$ be a sufficiently large integer, and let $\Z_n$ be the additive group modulo $n$.
    There exists a set of generator $S_1 \seq \Z_n$ of size $|S_1| \geq \sqrt{n/8} - O(n^{0.2625})$
    such that the graph $\Gamma_1=\Cayley(\Z_n, S_1)$ has cop number $c(\Gamma_2) \geq |S_1|/3 \geq
    \frac{\sqrt{n}}{3\sqrt{8}} - O(n^{0.2625}) \geq 0.1178 \sqrt{n} - O(n^{0.2625})$.
    \item \label{item:Cayley-p^k-even-k}
    Let $p$ be an odd prime power, and let $k \in \N$ be a positive even integer.
    Consider the abelian group $G_2 = \Z_p^k$ of order $n = p^k$.
    There exists a set of generators $S_2 \seq \Z_p^k$ of size $|S_2| = p^{k/2}+1$
    such that the graph $\Gamma_2 = \Cayley(G_2, S_2)$ has cop number $c(\Gamma_2) \geq |S_2|/3 > \sqrt{n}/3 > 0.3333 \sqrt{n}$.
    \item \label{item:Cayley-5-p^2}
    Let $p$ be an odd prime.
    Consider the abelian group $G_3 = \Z_5 \times \Z_p \times \Z_p$ of order $n = 5p^2$.
    There exists a set of generators $S_3 \seq G_3$ of size $|S_3| = 2p$
    such that the graph $\Gamma_3 = \Cayley(G_3, S_3)$ has cop number $c(\Gamma_3) = \ceil{(|S_3|+1)/2} = p+1 > \sqrt{n/5} > 0.4472\sqrt{n}$.
\end{enumerate}
\end{ftheorem}

We also prove that for \emph{any} abelian group $G$ of order $n$, such that $n$ is not divisible by 2 or 3,
there exists a set of generators $S \seq G$ such that the cop number of the corresponding Cayley graph $\Cayley(G, S)$ is lower bounded by $\Omega(n^{1/3})$.

\begin{ftheorem}\label{thm:root3-of-n-lb-for-all-cayley-graphs}
Let $G$ be any abelian group of order $n$ that contains no elements of order 2 or 3.
There exists a symmetric set of generator $S \seq G$ of size $|S| = \Omega(n^{1/3})$,
such that the Cayley graph $\Gamma = \Cayley(G, S)$ is connected and its cop number
is $c(\Gamma) \geq |S|/2 \geq \Omega(n^{1/3})$.
\end{ftheorem}

\section{Preliminaries}\label{sec:lemmas}

We prove our results by presenting a family of Cayley graphs $\Cayley(G,S)$ on $|G| = n$ vertices that are $K_{2,t}$-free for some value of $t$. This shows an example of a family of abelian Cayley graphs that achieves (up to a multiplicative constant) the bound of K\"{o}v\'{a}ri, S\'{o}s, and Tur\'{a}n~\cite{KovariST54}
for (a special case of) the Zarankiewicz problem,
stating that any $K_{2,3}$-free graph on $n$ vertices has at most $O(n^{1.5})$ edges.
Specifically, we describe examples of Cayley graphs on $n$ vertices with a generating set of size $\Omega(\sqrt{n})$ that are $K_{2,3}$-free.
Apply the following lemmas on these constructions in order to lower bound their cop number.

\begin{lemma}\label{lemma:K2t-free}
Fix $t \geq 3$. If $G=(V,E)$ is a $K_{2,t}$-free graph of minimum degree $\delta$, then $c(G) \geq \delta/t$.
\end{lemma}


\begin{lemma}\label{lemma:K23-triangle-free}
Fix $t \geq 3$. If $G=(V,E)$ be a $\{C_3, K_{2,t}\}$-free graph of minimum degree $\delta$, then $c(G) > (\delta+1)/(t-1)$.
\end{lemma}

Aigner and Fromme~\cite{fromme1984game} showed that if $G$ does not contain $C_3$ and $C_4$ then $c(G) \geq \delta$ holds.
Frankl~\cite{Frankl1987} showed that if $G$ does not contain $C_3$ and $K_{2,3}$ then $c(G) \geq (\delta+1)/2$.
Bonato and Burgess~\cite{BonatoBurgess13} also proved similar results.

\begin{proof}[Proof of \cref{lemma:K2t-free}]
    We prove that if the number of cops is less than $\delta/t$, then the robber can avoid the cops forever.
    Specifically, we prove the following claim.
    \begin{claim}\label{claim:K2t-free}
        For every $C \seq V$ of size $|C| < \delta/t$ and for every $v \in V \setminus C$ there is some $u \in N(v) \cup \{v\}$ that is not dominated by $C$,
        i.e., $u \notin \D(C)$, where $\D(C) = \cup_{c \in C} \D(c)$, and $\D(c) = \{c\} \cup N(c)$ are the vertices at distance at most 1 from $c$.
    \end{claim}
    \begin{proof}[Proof of \cref{claim:K2t-free}]
    Note that since $G$ is $K_{2,t}$-free, every $c \in C$ dominates at most $t$ neighbours of $v$,
    i.e., $|N(v) \cap \D(c)| \leq t$.%
    \footnote{
    If $c$ is not a neighbour of $v$, then it can dominate at most $t-1$ other neighbours of $v$.
    Otherwise it can dominate at most $t-1$ neighbours of $v$ other than itself.}
    Thus, the number of vertices in $\{v\} \cup N(v)$ that are dominated by $C$ is at most $|\{v\} \cup \left(\cup_{c \in C} (N(v) \cap \D(c)) \right)| \leq 1 + t|C|$.
    Therefore, if $|C| < \delta/t$, then the number of vertices in $\{v\} \cup N(v)$ that are dominated by $C$ is \emph{strictly less} than $1 + \delta \leq 1+\deg(v)$,
    and hence there is some $u \in N(v) \cup \{v\}$ that is not dominated by $C$.
    \end{proof}

    This implies that
    \begin{inparaenum}[(i)]
    \item
    in the initial round, given the locations $C \seq V$ of the cops, the robber can choose a vertex $u$
    so that $u \notin \D(C)$, and hence the cops cannot reach $u$ in the first round;
    \item
    in the subsequent rounds, given the locations $C$ of the cops, if the robber is in the vertex $v$
    then it can move to some $u \in N(v)$ so that $u \notin \D(C)$, and hence the cops capture it in the next round.
    \end{inparaenum}
\end{proof}

\begin{proof}[Proof of \cref{lemma:K23-triangle-free}]
    The proof of \cref{lemma:K23-triangle-free} is analogous to the above.
    The only difference is the analogue of \cref{claim:K2t-free} for $\{C_3, K_{2,t}\}$-free graphs.

    \begin{claim}\label{claim:K2t-triangle-free}
        For every $C \seq V$ of size $|C| \leq \delta/(t-1)$ and for every $v \in V \setminus C$ there is some $u \in N(v) \cup \{v\}$ that is not dominated by $C$,
        i.e., $u \notin \D(C)$.
    \end{claim}
    \begin{proof}[Proof of \cref{claim:K2t-triangle-free}]
    Note that since $G$ is $\{C_3, K_{2,t}\}$-free, every $c \in C$ dominates at most $t-1$ neighbours of $v$,
    i.e., $|N(v) \cap \D(c)| \leq t-1$.%
    \footnote{
    If $c$ is not a neighbour of $v$, then it can dominate at most $t-1$ other neighbours of $v$.
    Otherwise it can dominate no neighbour of $v$ other than itself.}
    Furthermore, since $G$ is $C_3$-free and $v \notin C$, if $v \in \D(c)$, then $c$ dominates no neighbour of $v$.
    Thus, the number of vertices in $\{v\} \cup N(v)$ that are dominated by $C$ is at most $(t-1)|C|$.
    Therefore, if $|C| \leq \delta/(t-1)$, then the number of vertices in $\{v\} \cup N(v)$ dominated by $C$ is at most $\delta \leq \deg(v)$,
    and hence $\exists u \in N(v) \cup \{v\}$ not dominated by $C$.
    \end{proof}
    The rest of the proof is exactly as in the proof of \cref{lemma:K2t-free}.
\end{proof}

\medskip
We will also need the following observation on Cayley graphs.
Let $\Gamma = \Cayley(G,S)$ be a Cayley graph with a symmetric set of generators $S$.
A 4-cycle (or a $K_{2,2}$) in $\Gamma$ is a collection of 4 edges corresponding to some generators $a,b,c,d \in S$ such that $a+b+c+d = 0$ (the elements are not necessarily distinct).
Observe first that any Cayley graph $\Gamma$ trivially contains a 4-cycle. Indeed, for any $s,s' \in S$
and any $d \in G$ and $d' = d+s+s'$ the vertices $\{d,d'\} \cup \{d+s,d+s'\}$ span a $K_{2,2}$.
Such 4-cycles in $\Gamma$ will be called ``trivial'', as they correspond to the trivial four tuple of elements in $S$ whose sum is zero, namely, $s+s'+(-s)+(-s') = 0$.
The following observation will be used several times in this paper.

\begin{observation}\label{obs:4-cycles-K-23-free}
  Let $\Gamma = \Cayley(G,S)$ be a Cayley graph with a symmetric set of generators $S$.
  If $\Gamma$ contains no non-trivial 4-cycles, then $\Gamma$ is $K_{2,3}$-free.
\end{observation}
\begin{proof}
    Suppose toward contradiction that $\Gamma$ contains a copy of $K_{2,3}$ with vertices $\{a,a'\}$ on one side and $b,b',b''$ on the other side.
    Then $S$ contains the generators $\{s_1 = b-a, s_1' = a'-b, s_2 = b'-a, s_2' = a'-b', s_3 = b''-a, s_3' = a'-b''\}$
    with $s_i \neq -s_i'$ for all $i = 1,2,3$.
    Observe that $s_1 + s_1' = s_2 + s_2' = s_3 + s_3'$, as all three are equal to $a'-a$.
    Therefore, since $S$ is symmetric, $\Gamma$ contains the 4-cycles corresponding to the sums $s_i+s_i'+(-s_j)+(-s_j') = 0$
    for $1 \leq i < j \leq 3$, and it is impossible for all of them to be trivial cycles.
\end{proof}

We will also need the following simple number theoretic lemma.

\begin{lemma}\label{lemma:equal-sums}
  Let $p \geq 3$ be a prime, and let $a,b,c,d$ be integers such that
  \begin{flalign*}
        a + b &\equiv c + d \bmod p \\
        a^2 + b^2 &\equiv c^2 + d^2 \bmod p \enspace.
  \end{flalign*}
  Then either ($a \equiv c \bmod p$ and $b \equiv d \bmod p$)  or ($a \equiv d \bmod p$ and $b \equiv c \bmod p$).
\end{lemma}
\begin{proof}
  Suppose that $a \not\equiv c \bmod p$, and therefore $b \not\equiv d \bmod p$.
  Then equation $a^2 + b^2 \equiv c^2 + d^2 \bmod p$ implies that
  $(a-c)(a+c) \equiv (d-b)(d+b) \bmod p$, and since $a-c  \equiv b-d \not\equiv 0 \bmod p$, it follow that $a+c \equiv b+d \bmod p$.
  this gives us the following system of equations.
  \begin{flalign*}
        a - c &\equiv d - b \bmod p \\
        a + c &\equiv d + b \bmod p \enspace.
  \end{flalign*}
  It is easy to see that all solutions must satisfy $a \equiv d \bmod p$ and $b \equiv c \bmod p$, as required.
\end{proof}

\section{Proofs of our results}\label{sec:proof}

In this section we prove \cref{thm:sqrt-lb} and \cref{thm:root3-of-n-lb-for-all-cayley-graphs}.

\subsection{Proof of Theorem \ref{thm:sqrt-lb} \cref{item:Cayley-Z_n}}
Fix a prime number $p \geq 5$.
For all $a \in \N$ define $s_a = (p^2+(a^2 \bmod p)p+a) \bmod 8p^2$, where $a^2 \bmod p$ is treated as an integer in $\{0,1,\dots,p-1\}$.
Note that $p^2 \leq s_a \leq 2p^2-2$ for all $0 \leq a \leq  p-1$ (where $s_a$ is treated as integer).%
\footnote{Indeed, for $0 \leq a \leq  p-2$ we have $s_a \leq p^2 + (p-1)p + a \leq 2p^2-2$,
and for $a = p-1$ we have $s_a = p^2 + p + (p-1) \leq 2p^2 - 2$.}. Define the sets $S^+=\{s_a: a \in \{0,1,\dots,(p-1)/2\}\}$, $S^-=-S^+$, and let $S=S^+ \cup S^-$.

\begin{lemma}\label{lemma:z-n-C-3-C-4-free}
The set $S$ satisfies the following properties.
\begin{enumerate}
  \item $s_a \neq s_{a'}$ for all $0 \leq a < a' \leq (p-1)/2$. In particular, $|S| = p+1$. \label{item:distinct-s-a}
  \item For any $s_1,s_2,s_3 \in S$ it holds that $2 \leq |s_1 + s_2 + s_3| < 6p^2$. \label{item:3cycle}
  \item Let $s_1 \geq s_2 \geq s_3 \geq s_4$ be elements in $S$ such that $s_1 + s_2 + s_3 + s_4 = 0$.
  Then $s_1 = -s_4$ and $s_2 = -s_3$. \label{item:4cycle}
\end{enumerate}
\end{lemma}
\begin{proof}
For \cref{item:distinct-s-a} observe that all $s_a$'s are distinct, as they are distinct modulo $p$,
and analogously, all elements in $S^-$ are distinct. It is also clear that $S^+$ and $S^-$ are disjoint.

For \cref{item:3cycle}, let $a_1,a_2,a_3 \in \{0,1,\dots,(p-1)/2\}$ be such that $s_i \in \{\pm s_{a_i}\}$ for all $i=1,2,3$.
Suppose first that $s_1,s_2,s_3 \in S^+$, i.e., $s_i = s_{a_i}$ for all $i = 1,2,3$.
Then the sum $s_1+s_2+s_3 = s_{a_1}+s_{a_2}+s_{a_3}$ is between $3p^2$ and $3(2p^2-2) < 6p^2$.
Similarly, if $s_1,s_2,s_3 \in S^-$, then $s_i = -s_{a_i}$ for all $i = 1,2,3$, and hence $-6p^2 < -3(2p^2-2) \leq s_1+s_2+s_3 \leq -3p^2$, as required.

Next, consider the case where two elements are in $S^+$ and one is in $S^-$.
Then, the sum of the corresponding elements is $s_{a_1} +s_{a_2} - s_{a_3} \geq p^2 + p^2 - (2p^2-2) \geq 2$, as required.
The case of one element in $S^+$ and two elements in $S^-$ is similar.

For \cref{item:4cycle} consider the cases based on how many elements $s_i$'s are in $S^+$ or in $S^-$.
\begin{itemize}
  \item If all four elements are in $S^+$ or all four elements are in $S^-$, then their sum cannot be zero.
  \item If three elements are in $S^+$ and one element is in $S^-$, then their sum cannot be zero, as $s_1+s_2+s_3+s_4 \geq 3 p^2 - (2p^2-2) = p^2 + 2 > 0$.
  Similarly, if three elements are in $S^-$ and one element is in $S^+$.
  \item Finally, consider the case where $s_1,s_2 \in S^+$ and $s_3,s_4 \in S^-$.
  Let $a_1,a_2,a_3,a_4 \in \{0,1,\dots,(p-1)/2\}$ be such that $s_1 = s_{a_1}, s_2 = s_{a_2}, s_3 = -s_{a_3}, s_4 = -s_{a_4}$,
  and hence $s_{a_1}+s_{a_2} = s_{a_3} + s_{a_4}$.
  Observe that by definition of $s_{a_i}$ this implies
  \begin{flalign*}
        a_1 + a_2 &\equiv a_3 + a_4 \bmod p \\
        a_1^2 + a_2^2 &\equiv a_3^2 + a_4^2 \bmod p \enspace.
  \end{flalign*}
By \cref{lemma:equal-sums} all solutions to this system of equations satisfy $(a_1 = a_3, a_2 = a_4)$  or $(a_1 = a_4, a_2 = a_3)$.
Therefore, the assumption $s_1 \geq s_2 \geq s_3 \geq s_4$ implies that $a_1 = a_4$ and $a_2 = a_3$.
This completes the proof of \cref{lemma:z-n-C-3-C-4-free}. \qedhere
\end{itemize}
\end{proof}

We are now ready to prove \cref{item:Cayley-Z_n} of \cref{thm:sqrt-lb}. Fix an integer $n$.
Baker, Harman, and Pintz proved in~\cite{BakerHP01} that for all sufficiently large $x$, there exists a prime between $x- x^{0.525}$ and $x$.
In particular, for $x = \sqrt{n/8}$ there exists a prime $p$ such that $\sqrt{n/8} - (n/8)^{0.2625} \leq p \leq \sqrt{n/8}$.

Let $S_1=S \cup\{-1,1\}$ be the set of generators in $\Z_n$, where $S = S^+\cup S^-$ is as above.
Note that $|S_1| \geq |S| = 2p$, and $\Gamma_1$ is connected since $S_1$ is a generating set of $\Z_n$ as $1 \in S_1$.

\begin{claim}\label{claim:z-n-C-3-K-23-free}
    The Cayley graph $\Gamma_1=\Cayley(\Z_n, S_1)$ is $\{C_3,K_{2,4}\}$-free.
\end{claim}

\begin{proof}
By definition, $\Gamma_1$ contains a $C_3$ if and only if there are three elements in $S_1$ whose sum is $0$ in $\Z_n$.
It follows from \cref{lemma:z-n-C-3-C-4-free} that he sum of any 3 elements in $S$ is between $2$ and $6p^2$, and hence cannot be $0$ in $\Z_n$.
It is also easy to see there are no $s_1,s_2 \in S$ such that $|s_1-s_2| = 1$, and hence, $\Gamma_1$ in $C_3$-free.

Next we show that $\Gamma_1$ is $K_{2,4}$-free.
Recall that a 4-cycle in $\Gamma_1$ is a collection of four edges corresponding to four elements $a,b,c,d \in S_1$ such that $ab+c+d=0$.
Also, recall that a 4-cycle is called ``trivial'' if the sum is of the form $s+s'+(-s)+(-s') = 0$.

Note that if $s_1+s_2+s_3+s_4 \equiv 0 \bmod n$, then $s_1+s_2+s_3+s_4 = 0$ as an integer,
because $|s| < 2p^2$ for all $s \in S_1$ and $n \geq 8p^2$.
Therefore, by \cref{lemma:z-n-C-3-C-4-free} \cref{item:4cycle} any nontrivial 4-cycle in $\Gamma_1$ must contain an edge $(d,d+s)$
such that $s \in \{-1,1\}$. Furthermore, by \cref{lemma:z-n-C-3-C-4-free} \cref{item:3cycle} it follows that
any nontrivial 4-cycle in $\Gamma_1$ must contain at least two such edges.
This implies that $\Gamma_1$ is $K_{2,4}$-free.
\end{proof}
By applying \cref{lemma:K2t-free}, we get $c(\Gamma_1)\geq |S_1|/3= \frac{p}{3} \geq \frac{\sqrt{n}}{3\sqrt{8}} - O(n^{0.2625})$, as required.

\subsection{Proof of Theorem \ref{thm:sqrt-lb} \cref{item:Cayley-p^k-even-k}}

For the proof we consider the finite field $\GF(p^k)$, and treat $\Z_p^{k}$ as the additive group of $\GF(p^k)$.
Let $q = p^{k/2}$. Recall that $p$ is an odd prime power and $k$ is even, and hence $q$ is an odd prime power.
Define the set of generators to be
\begin{equation*}
    S_2 = \{s\in \GF(p^k) : s^{q+1}=1\}
    \enspace,
\end{equation*}
where the power $s^{q+1}$ is in the field $\GF(q^2)$.
Note that since $q$ is odd, $S_2$ is, indeed, symmetric as for all $s \in S_2$ we have $(-s)^{q+1}=(-1)^{q+1} \cdot s^{q+1} = 1$, and hence $-s \in S_2$.
Also note that $|S_2|=q+1$, since the multiplicative group of $\GF(p^k)$ is a cyclic group of order $p^k-1 = q^2-1$, and hence contains a generating element $\alpha$ of order $q^2-1 = (q+1)(q-1)$. Therefore $S_2 = \{\alpha^{(q-1) i} : i \in \{0,1,2,\dots,q \}\}$.

\begin{claim}\label{claim:G-2-K-23-free}
    The graph $\Gamma_2 = \Cayley(G_2, S_2)$ is $K_{2,3}$-free.
    In particular, for all $a_1,b_1,a_2,b_2 \in S_2$ such that $a_1 \neq -b_1$, $a_2 \neq -b_2$, and $\{a_1,b_1\} \neq \{a_2, b_2\}$
    it holds that $a_1+b_1 \neq a_2+b_2$.
\end{claim}

\begin{proof}
If $d_1, d_2$ are distinct elements of $\GF(q^2)$, then the number of vertices in $\Gamma_2$ adjacent to both $d_1$, and $d_2$ is equal to the number of solutions of the below system of equations.
\begin{flalign*}
    (x-d_1)^{q+1} &= 1 \\
    (x-d_2)^{q+1} &= 1 \enspace,
\end{flalign*}
or equivalently
\begin{flalign*}
    (x-d_1) (x^q-d_1^q) &= 1 \\
    (x-d_2) (x^q-d_2^q) &= 1 \enspace.
\end{flalign*}
This is a special case of system of equations (4) in \cite{KollarRS96} ($K=\GF(p^k), t=2, a_{ij}=d_j^{q^{i-1}}, x_i=x^{q^{i-1}}, b_j=1$).
Thus, according to Theorem 3 in \cite{KollarRS96}, the system of equations has at most $t!=2$ solutions.
Therefore, the Cayley graph $\Cayley(G_2, S_2)$ is $K_{2,3}$-free.

For the ``in particular'' part, note that if we had two distinct pairs $\{a_1,b_1\}$ and $\{a_2, b_2\}$ with $a_1 \neq -b_1$ and $a_2 \neq -b_2$ such that $a_1+b_1 = a_2+b_2$,
then we would get a copy of $K_{2,3}$ in $\Gamma_2$ with the vertices $\{d_1 = 0, d_2 = a_1+b_1\}$ on one side and $\{a_1,b_1,a_2\}$ on the other side.
\end{proof}

Finally, observe that $S_2$ is a generating set for $\Z_p^k$.
Indeed, by the ``in particular'' part of \cref{claim:G-2-K-23-free} the set $S_2$ spans at least $\binom{|S_2|}{2}=\binom{q+1}{2} > q^2/2$ elements of $G$,
as for any pair $a,b \in S_2$ with $a \neq -b$ produces a different sum in $G_2$.
Since the number of elements spanned by $S_2$ divides $q^2$, it must be the case that $S_2$ generates the entire group $\Z_p^k$, and hence $\Cayley(\Z_p^k, S_2)$ is connected.

Using \cref{lemma:K2t-free}, we conclude that $c(\Gamma_2) \geq |S_2|/3=(q+1)/3 > \sqrt{n}/3$, as required.

\subsection{Proof of Theorem \ref{thm:sqrt-lb} \cref{item:Cayley-5-p^2}}
Consider the abelian group $G_3 = \Z_5 \times \Z_p \times \Z_p$ of order $n = 5p^2$.
Define the set of generators $S_3 = \{(1,a,a^2): a \in \Z_p\} \cup \{(-1,-a,-a^2): a \in \Z_p\}$, where $a^2$ is taken modulo $p$.
Note that $S_3$ is indeed a symmetric set of size $|S_3| = 2p$.

Let $\Gamma_3 = \Cayley(G_3, S_3)$ be the corresponding Cayley graph. We show below that $\Gamma_3$ is $\{C_3, K_{2,3}\}$-free,
and hence by \cref{lemma:K23-triangle-free} we conclude that $c(\Gamma_3) \geq |S_3|/2 =p$, as required.

\begin{claim}\label{claim:G-3-C-3-K-23-free}
    The graph $\Gamma_3$ is connected and $\{C_3, K_{2,3}\}$-free.
\end{claim}
\begin{proof}
Observe that $\Gamma_3$ has no triangles because there are no three elements in $S$ whose sum is 0 in the first coordinate.

Next we claim that $\Gamma_3$ is $K_{2,3}$-free.
This is done by proving that $\Gamma_3$ contains no non-trivial 4-cycles.
Indeed, let $s_1,s_2,s_3,s_4 \in S_3$ be four generators such that $s_1+s_2+s_3+s_4 = 0$ in $G_3$,
By looking at the first coordinate (to $\Z_5$), it must be the case that two of the $s_i$'s are in $\{(1,a,a^2): a \in \Z_p\}$
and two are in $\{(-1,-a,-a^2): a \in \Z_p\}$.
Assume without loss of generality that $s_1 = (1,a,a^2)$, $s_2 = (1,b,b^2)$, $s_3 = (-1,-c,-c^2)$, $s_4 = (-1,-d,-d^2)$ for some $a,b,c,d \in \Z_p$.
Therefore, if $s_1+s_2+s_3+s_4 = 0$, then $a+b \equiv c+d \bmod p$ and $a^2+b^2 \equiv c^2+d^2 \bmod p$.
Therefore, by \cref{lemma:equal-sums} we either have ($a = c$ and $b = d$)  or ($a = d$ and $b = c$).
Therefore, $\Gamma_3$ contains only trivial 4-cycles, as required.
Therefore, by \cref{obs:4-cycles-K-23-free} the Cayley graph $\Gamma_3$ is $K_{2,3}$-free.

In order to see that $\Gamma_3$ is connected, note that the elements spanned by $S_3$ form a subgroup of $G_3$, and hence $5p^2$ is divisible by $|\sp(S_3)|$.
Since $\Gamma_3$ contains no non-trivial 4-cycles, it follows that the number of elements spanned by $S_3$ is $|\sp(S_3)| \geq |\{s+s' :s,s' \in S_3, s'\neq s\}| \geq \binom{|S_3|}{2} \geq \binom{2p}{2} = p(2p-1)$, and hence $S_3$ spans the entire group $G_3$.
\end{proof}

By \cref{lemma:K23-triangle-free} the cop number of $\Gamma_3$ is $c(\Gamma_3) \geq (|S_3|+1)/2 \geq (2p+1)/2$.
On the other hand, according to \cite[Theorem 1]{Frankl1987} we have $c(\Gamma_3)\leq \ceil{(|S_3|+1)/2} = \ceil{(2p+1)/2} = p+1$.
Therefore, $c(\Gamma_3)=p+1$.

\subsection{Proof of \cref{thm:root3-of-n-lb-for-all-cayley-graphs}}

Let $G$ be an abelian group of order $n$ such that $G$ has no elements of order $2$ or $3$.
We construct a generating set $S \seq G$ using \cref{alg:greedy-s}.
Before describing the algorithm we make the following notation.
\begin{notation}
    For a subset $S \seq G$ let
    $\Forbidden_1(S) = \{a+b+c : a,b,c\in S\}$, $\Forbidden_2(S) = \{ a : \exists b,c \in S \text{ s.t. } b+c+a+a=0\}$, and $\Forbidden_3(S) = \{a : a+a+a \in S \}$.
    Define $\Forbidden_S = \Forbidden_1(S) \cup \Forbidden_2(S) \cup \Forbidden_3(S)$.
\end{notation}

\begin{claim}\label{claim:arbitrary-g-K23-free}
    Let $S \seq G$ be a symmetric set,
    and suppose that $S$ has no non-trivial 4-cycles.
    Then, for any $s^* \in G \setminus \Forbidden_S$
    the set $S \cup \{s^*,-s^*\}$ has no non-trivial 4-cycles.
\end{claim}

\begin{proof}
Observe first that $S \seq \Forbidden_1(S)$, as for any $s \in S$ we have $s = s+s+(-s) \in \Forbidden_1(S)$.
In particular $S \seq \Forbidden_S$, and thus if $S \cup \{s^*,-s^*\}$ contains a non-trivial 4-cycle $a+b+c+d = 0$,
then at least one of the elements must be in $\{s^*,-s^*\}$.

Note that for any three elements $a,b,c \in S$ we have $a+b+c$ in $\Forbidden_1(S) \seq \Forbidden_S$,
and $s^*, -s^* \notin \Forbidden_S$. Therefore $S \cup \{s^*,-s^*\}$ does not contains a non-trivial 4-cycle with exactly one element in $\{s^*,-s^*\}$.

Suppose now that two of the elements $\{a,b,c,d\}$ are in $\{s^*,-s^*\}$. Since the 4-cycle is non-trivial,
it must be that the two of the elements are equal. Without loss of generality suppose that $a=b=s^*$.
But then $s^* \in \Forbidden_2(S)$, and hence $a+b+c+d=0$ cannot be a non-trivial 4-cycle with two edges outside $S$.

Similarly, if three of the elements $a,b,c,d$ belong to $\{s^*,-s^*\}$, we may assume without loss of generality that $a=b=c=s^*$.
But this implies that $s^* \in \Forbidden_3(S)$, and hence  $a+b+c+d=0$ cannot be a non-trivial 4-cycle with three edges outside $S$.

Finally, since $G$ does not contain elements of order 2, it is impossible that all four elements $a,b,c,d$ belong to $\{s^*,-s^*\}$.

This completes the proof of \cref{claim:arbitrary-g-K23-free}
\end{proof}

We are now ready to describe the algorithm.

\begin{algorithm}
\caption{Constructing a generating set $S$ of a group $G$}\label{alg:greedy-s}
\begin{algorithmic}
    \State $S_0 \gets$ a minimal generating set of $G$
    \State $S \gets S_0 \cup -S_0$
    \While {$G \neq \Forbidden_S$}
        \State Choose  an arbitrary element $s \in G\setminus \Forbidden_S$
        \State $S \gets S \cup \{-s,s\}$
    \EndWhile
    \State \Return $S$
\end{algorithmic}
\end{algorithm}

For the analysis observe first that in the end of each iteration we have $|\Forbidden_S| \leq |S|^3 + |S|^2 + |S|$.
Indeed, we have
\begin{inparaenum}[(i)]
\item $|\Forbidden_1| \leq |S|^3 = k^3$,
\item $|\Forbidden_2| \leq |S|^2 = k^2$, as $G$ has no elements of order $2$, and
\item $|\Forbidden_3| \leq |S| = k$, as $G$ has no elements of order $3$.
\end{inparaenum}
Therefore, since the algorithm ends when $|\Forbidden_S| = n$, it follow that the output is a set $S$ of size $\Omega(n^{1/3})$.

Note first that since $S$ contains a generating set of $G$, the graph $\Gamma = \Cayley(G,S)$ is connected.
Also, note that since $S_0$ is a minimal generating set of $G$, the set $S$ before the loop contains no non-trivial 4-cycles.
Indeed, it is not difficult to see that if $G$ contains no elements of order 2, and $S_0 \cup -S_0$ contains a non-trivial four cycle $a+b+c+d = 0$,
then $S_0$ contains a strict subset generating $G$.

By \cref{claim:arbitrary-g-K23-free} in each iteration of the algorithm, $S$ does not contain a non-trivial 4-cycles in any iteration,
and hence, by \cref{obs:4-cycles-K-23-free} in the end of the algorithm the graph $\Gamma = \Cayley(G,S)$ is $K_{2,3}$-free.
Therefore, by \cref{lemma:K2t-free} $c(\Gamma) \geq |S|/3 \geq \Omega(n^{1/3})$, as required.

\section{Final Remarks and Open Problems}
We showed in \cref{thm:sqrt-lb} several classes of Meyniel extremal Cayley graphs.
Our \cref{thm:root3-of-n-lb-for-all-cayley-graphs} shows a weaker result for general groups, namely, that any group satisfying certain mild conditions has a Cayley graph of order $\Omega(n^{1/3})$.
This raises the following natural question.

\begin{question}
    Is it true that any group $G$ has a Cayley graph that is Meyniel extremal?
\end{question}

Also, Pralat~\cite{Pralat10} showed a family of graphs on $n$ vertices whose cop number $\geq \sqrt{n/2} \cdot (1-o(1))$.
It would be interesting to find a family of Cayley graphs matching with the same parameters.

\bibliographystyle{plain}
\bibliography{mybib}

\end{document}